\documentclass[]{article}

\usepackage{amsfonts}
\usepackage{amsmath}
\usepackage{amsthm}
\usepackage{mathtools}
\usepackage{float}
\usepackage{enumitem}
\usepackage{thmtools, thm-restate}

\usepackage{tikz}

\setlist[enumerate]{label*=\arabic*.}

\theoremstyle{plain}
\newtheorem{theorem}{Theorem}[section]
\newtheorem{corollary}{Corollary}[theorem]
\newtheorem{lemma}[theorem]{Lemma}
\newtheorem{proposition}[theorem]{Proposition}

\theoremstyle{definition}
\newtheorem{definition}[theorem]{Definition}
\newtheorem{notation}{Notation}

\newtheorem{example}{Example}[section]

\theoremstyle{remark}
\newtheorem*{remark}{Remark}

\newcommand{\Aut}{\mathrm{Aut}}
\newcommand{\Inn}{\mathrm{Inn}}
\newcommand{\Out}{\mathrm{Out}}
\newcommand{\Fix}{\mathrm{Fix}}
\newcommand{\Stab}{\mathrm{Stab}}
\newcommand{\Hyp}{\mathrm{Hyp}}

\newcommand{\maxc}[1]{\langle | #1 | \rangle}

\title{Fixed Points of Automorphisms of Torus Knot Groups}
\author{Oli Jones \footnote{Email: oj2003@hw.ac.uk Address: School of Mathematical and Computer Sciences, Heriot-Watt University, Edinburgh, Scotland, EH14 4AS}}
\date{}

\begin{document}
\usetikzlibrary{arrows,positioning}
\maketitle

\begin{abstract}
We completely classify fixed point subgroups in Torus Knot Groups, that is groups of the form $G_{p,q} = \langle x , y | x^p = y^q \rangle$. We not only give the isomorphism type, but also the explicit generators for the fixed point subgroup of each automorphism of $G_{p,q}$. Our main tool is an action of $\Aut(G_{p,q})$ on the Bass-Serre tree of $G_{p,q}$ which is compatible with the original action, in the sense that it extends the original action of $G_{p,q}$ on its Bass-Serre tree.
\end{abstract}



\section{Introduction}

Given an automorphism $\varphi$ of a group $G$, its set of fixed points, $$\Fix(\varphi) = \{g \in G | \varphi(g) = g\},$$ is a subgroup of $G$. Fixed points were first prominently studied in free groups; in 1975 Dyer and Scott published the famous Scott Conjecture \cite{DS75}, that the rank of a fixed point subgroup in a finitely generated free group was bounded above by the rank of the ambient group. This conjecture was proven by Bestvina and Handel in 1992 \cite{BH92}. Fixed points have since been explored in various contexts, including hyperbolic groups \cite{N92}; groups acting on trees with finite edge stabilisers \cite{S02}; relatively hyperbolic groups \cite{MO11}; and right-angled Coxeter and Artin groups \cite{F21}.

The purpose of this paper is to fully classify fixed point subgroups in torus knot groups, that is groups with the presentation $G_{p,q} = \langle x, y | x^p = y^q \rangle$, where $p, q \in \mathbb{Z}_{\geq 2}$. See Example 1.24 of \cite{H01} for the perspective on these groups as the fundamental group of a torus knot. It is also notable that this class includes the odd dihedral Artin groups; the Artin group defined by an edge labelled $2n+1$ is isomorphic to $G_{2, 2n+1}$ \cite{CHR20}. 

Torus knot groups split as an amalgamated product $\mathbb{Z} *_{\mathbb{Z}} \mathbb{Z}$. As such they act on a tree, known as the Bass-Serre tree, with cyclic edge and vertex stabilisers. Our primary tool for understanding fixed points comes from \cite{GHMR99}, which takes an action of $G_{p,q}$ on its Bass-Serre tree, and produces a \emph{compatible} action of $\Aut(G_{p,q})$ on the same tree. Compatibility will be formally defined in Definition \ref{compatibleaction}, but essentially means the action of an inner automorphism $\iota_g: x \mapsto gxg^{-1}$ will be the same as the action of $g$. In Section \ref{compatiblesection} we explore how one can recover information about fixed points from a compatible action. 

Section \ref{compatiblesection} is written in the general setting of a group $G$ and its automorphism group acting compatibly on a set $X$. Such compatible actions have been produced for other groups, for instance mapping class groups \cite{I02}, Higman's group \cite{M17}, graph products \cite{GM18}, and certain Artin groups \cite{V23}. For any group acting minimally on a tree which is not a line, a subgroup of its automorphism group has such a compatible action \cite{A21} (in our case this subgroup is the whole automorphism group). It is likely that these actions could be exploited using a similar strategy to the current paper to understand fixed points in those groups.

In order to state the main theorem we need to describe the automorphisms of torus knot groups. Throughout the paper, we will use $\iota_g$ to denote conjugation by $g$. We will write $[\varphi]$ for the equivalence class of an automorphism $\varphi$ in the outer automorphism group $\Out(G_{p,q})$.

If $p \neq q$, then $\Out(G_{p,q}) \cong C_2$, and if $p=q$, $\Out(G_{p,q}) \cong C_2 \times C_2$ \cite{GHMR99}. In both cases there is the non-inner automorphism $\tau: G_{p,q} \rightarrow G_{p,q}, x \mapsto x^{-1}, y \mapsto y^{-1},$ which we will sometimes call the global inversion. This generates a $C_2$ factor in $\Out(G_{p,q})$ in both cases. When $p = q$, $G_{p,q}$ also has the non-inner automorphism $\sigma: G_{p,p} \rightarrow G_{p,p}, x \mapsto y, y \mapsto x$, which generates the other cyclic factor. This means that any automorphism of $G_{p,q}$ where $p \neq q$ can be written as $\iota_g \tau^r$ and any automorphism of $G_{p,p}$ can be written as $\iota_g \tau^r \sigma^s$, where $g \in G_{p,q}$ and $r, s \in \{0,1\}$. 

With this established we are ready to state the main theorem, which fully classifies the fixed point subgroups of torus knot groups up to isomorphism, further giving explicit generators in each case. In the following, $\maxc{g}$ denotes the maximal cyclic subgroup containing $g$. This is well defined for non-central elements of torus knot groups (Lemma \ref{maxc}). We take care to only use this notation when it is well defined.

\begin{theorem}\label{maintheorem1}
Let $G_{p,q} = \langle x, y | x^p = y^q \rangle$ be a torus knot group (with no restriction on $p, q \in \mathbb{Z}_{\geq 2}$). Then the fixed points of non-trivial automorphisms in the outer automorphism class $[id]$ are as follows.

\begin{enumerate}
\item If $\iota_g$ is finite order, then $\Fix(\iota_g) = \maxc{g} \cong \mathbb{Z}$ (Lemma \ref{ellipticcentralizers}).
\item If $\iota_g$ is infinite order, then we may assume $g$ acts directly (see Definition \ref{directdefn}), and $\Fix(\iota_g) = \maxc{g} \times \langle x^p \rangle \cong \mathbb{Z}^2$ (Lemma \ref{hyperboliccentralizers}).
\end{enumerate}

The fixed points of non-trivial automorphisms in $[\tau]$ are given by the following.
\begin{enumerate}
\item If $\iota_g \tau$ is finite order, then there exists $h \in G$ and $r \in \mathbb{Z}$ such that $g = hw^rh^{-1}$, where $w \in \{x,y\}$ (Lemma \ref{trueformlemma}). The structure of the fixed point subgroup of $\iota_g \tau$ may depend on the parity of $r$.

\begin{enumerate}
\item If $p$, $q$ and $r$ are even, then $\Fix(\iota_g \tau) = \langle hw^{\frac{r}{2}}x^{\frac{p}{2}}y^{-\frac{q}{2}}w^{-\frac{r}{2}}h^{-1} \rangle \cong \mathbb{Z}$, where $w \in \{x,y\}$ is as above (Lemma \ref{hyperbolicfixedpoints}).
\item If $p$, $q$ and $r$ are not all even, then $\Fix(\iota_g \tau) = \{1\}$ (Lemma \ref{ellipticfixedpoints} and Lemma \ref{hyperbolicfixedpoints}).
\end{enumerate}

\item If $\iota_g \tau$ is infinite order, then $\Fix(\iota_g) = \maxc{g\tau(g)} \cong \mathbb{Z}$ (Lemma \ref{algebraictau}).

\end{enumerate}
\end{theorem}

In the case $p = q$ there are further automorphisms not present in Theorem \ref{maintheorem1}, coming from the extra non-inner automorphism $\sigma$. The automorphisms in the classes $[id]$ and $[\tau]$ are covered by Theorem \ref{maintheorem1}.

\begin{theorem}\label{maintheorem2}
Let $G_{p,p} = \langle x, y | x^p = y^p \rangle$. The fixed points of non-trivial automorphisms in the outer automorphism class $[\sigma]$ are as follows.

\begin{enumerate}
\item If $\iota_g \sigma$ is finite order then $\Fix(\iota_g\sigma) = \langle x^p \rangle \cong \mathbb{Z}$ (Lemma \ref{geometricsigma}).
\item If $\iota_g \sigma$ is infinite order then $\Fix(\iota_g \sigma) = C(g\sigma(g)) \cong \mathbb{Z}^2$ (Lemma \ref{algebraicsigma}). The generators are $x^p$ and $\hat{g\sigma(g)}$, where $\hat{g\sigma(g)}$ acts hyperbolically along the same axis as $g\sigma(g)$ with minimal translation length (Lemma \ref{hyperboliccentralizers}).

\end{enumerate}

The fixed points of non-trivial automorphisms in the class $[\sigma\tau]$ are given by the following.

\begin{enumerate}

\item If $\iota_g \sigma \tau$ is finite order then $\Fix(\iota_g\sigma) = \{1\}$ (Lemma \ref{geometricsigma}).

\item If $\iota_g \sigma \tau$ is infinite order then $\Fix(\iota_g \sigma \tau) = \maxc{g(\sigma\tau(g))} \cong \mathbb{Z}$ (Lemma \ref{algebraicsigma}). 
\end{enumerate}

\end{theorem}

The paper is organised as follows. In Section 2 we discuss isometries of trees and the Bass-Serre theory of amalgamated products, and introduce the action of $\Aut(G_{p,q})$ on the Bass-Serre tree of $G_{p,q}$. In Section 3 we define a compatible action, and explore the relationship between compatible actions and fixed points in the general setting.

We begin computing fixed points in Section 4, by finding the fixed points of inner automorphisms. Section 5 finds the fixed points of all automorphisms of $G_{p,q}$ in the outer automorphism class $[\tau]$. When $p \neq q$ this completes our understanding. The infinite order automorphisms are dealt with relatively quickly in Lemma \ref{algebraictau} by an algebraic argument. For the finite order automorphisms, we make use of the compatible action and results from Section 3 to find the fixed points geometrically. In Section 6 we deal with the case $p=q$, where there is an extra generator in the outer automorphism group. The fixed points of the extra automorphisms, in the classes $[\sigma]$ and $[\sigma\tau]$, are handled in much the same way as Section 5.

\section*{Acknowledgements}

I am grateful to my supervisor Laura Ciobanu for her feedback and support. I thank Alan Logan for helpful discussions around compatible actions, in particular noticing the result of Lemma \ref{compatibilitylemma}. Finally, thank you to Nicolas Vaskou for his comments on the manuscript.

\section{Preliminaries}

\subsection{Isometries of Trees}\label{isosoftrees}

A (simplicial) tree is a connected, contractible graph. Our trees will have finite valence at every vertex. We equip each tree $T$ with a metric $d_T$, such that the distance between two vertices is the number of edges between them, and each edge is isometric to the unit interval.

Let $\gamma$ be a \emph{simplicial isometry} of a tree $T$, that is an isometry that sends vertices to vertices and edges to edges. From here we will write isometry to mean simplicial isometry. For $\gamma$ an isometry of a tree $T$ we define $$l(\gamma) := \inf_{x \in T} d_T(x, \gamma x),$$ and $$T^\gamma := \{x \in T | d(x, \gamma x) = l(\gamma)\}.$$

If $l(\gamma) = 0$, $\gamma$ is called \emph{elliptic}. Otherwise $\gamma$ is called \emph{hyperbolic}. 

\begin{proposition}\textbf{\cite[Proposition 24]{SS12}}\label{geometricfixsubtree}
The set $T^\gamma$ is a non-empty connected subtree. If $l(\gamma) > 0$, that is $\gamma$ has no fixed points, then $T^\gamma$ is a bi-infinite line, and $\gamma$ acts as translation of amplitude $l(\gamma)$ along this axis.
\end{proposition}

Proposition \ref{geometricfixsubtree} implies that any elliptic isometry $\gamma$ has fixed points, namely the set $T^\gamma \neq \phi$. If $n \in \mathbb{Z}_{>0}$, notice that $l(\gamma^n) = n ~ l(\gamma)$. In particular, proper powers of elliptic elements are elliptic, and proper powers of hyperbolic elements are hyperbolic.

Notice also that $l(\gamma) \in \mathbb{Z}_{\geq 0}$ for all isometries $\gamma$, because the isometries are assumed to be simplicial.

\subsection{Bass-Serre Theory}

The groups $G_{p,q} = \langle x, y | x^p = y^q \rangle$ split as amalgamated products $\langle x \rangle \ast_{x^p = y^q} \langle y \rangle$. As such, there is a corresponding graph $T$, known as the Bass-Serre tree, given by the following definition; this graph has properties as in Lemma \ref{bassserretheory}. In particular it is a tree, justifying the name Bass-Serre tree. Note that although the tree depends on the parameters $p$ and $q$, we will suppress this in the notation.

\begin{definition}
The Bass-Serre tree, $T$, for the torus knot group $G_{p,q}$, is the graph with vertices $v_{g\langle x \rangle}$ or $v_{g\langle y \rangle}$ and edges $e_{g\langle x^p \rangle}$, where $g \in G$, identifying simplices when the corresponding cosets are equal. The endpoints of $e_{g\langle x^p \rangle}$ are defined as $v_{g\langle x \rangle}$ and $v_{g\langle y \rangle}$.

\end{definition}

The stabiliser of a simplex $s$ will be denoted by $G_s$.

\begin{lemma}\textbf{\cite[Theorem 7]{SS12}}\label{bassserretheory}
The graph $T$ is a tree. $G_{p,q}$ acts on $T$ on the left without inversions by $hv_{g\langle x \rangle} := v_{hg\langle x \rangle}$, and likewise for edges. A fundamental domain is given by $v_{\langle x \rangle}$, $v_{\langle y \rangle}$ and $e_{\langle x^p \rangle}$.

For each $g \in G_{p,q}$, $G_{v_{g\langle x \rangle}} = g\langle x \rangle g^{-1}$, $G_{v_{g\langle y \rangle}} = g\langle y \rangle g^{-1}$ and $G_{e_{g\langle x^p \rangle}} = \langle x^p \rangle$.
\end{lemma}

The tree $T$ is known as the Bass-Serre tree of $G_{p,q}$ corresponding to the splitting as an amalgam. That the action is without inversions means there is no edge which has its midpoint fixed but its endpoints exchanged.

From here we will usually denote $v_{g\langle x \rangle}$ by $gv_{\langle x \rangle}$.

We write $Z(G_{p,q})$ for the centre of $G_{p,q}$.

\begin{lemma}\label{centre}
The centre of $G_{p,q}$ is infinite cyclic and generated by $x^p = y^q$. In particular, the centre is the stabiliser of each edge in $T$.
\end{lemma}
\begin{proof}
Firstly, $x^p = y^q$ commutes with $x$ and $y$, so $\langle x^p \rangle \leq Z(G_{p,q})$.

For the reverse inclusion, consider $\phi: G_{p,q} \rightarrow \langle x | x^P \rangle * \langle y | y^q \rangle$, the quotient by $\langle x^p \rangle$. If $g \in Z(G_{p,q})$ then, since $\phi$ is surjective, $\phi(g) \in Z(\langle x | x^p \rangle * \langle y | y^q \rangle) = 1$. So $Z(G_{p,q}) \leq Ker(\phi) = \langle x^p \rangle$. So $Z(G_{p,q}) = \langle x^p \rangle$.

Since $G_{e_{g\langle x^p \rangle}} = \langle x^p \rangle$ for each edge $e_{g\langle x^p \rangle}$, the statement is proved.
\end{proof}

For a bi-infinite line $\omega$ in the tree, recall $\Stab(\omega)$ is the setwise stabiliser. We now study a subgroup which should be thought of as elements that act along the line, preserving its orientation. We write $\Hyp(\omega)$ for the subgroup of $\Stab(\omega)$ comprised of elements which fix a point in $\omega$ if and only if they (pointwise) fix $\omega$. That is, we exclude elliptic elements of $\Stab(\omega)$ acting non-trivially on $\omega$ from $\Hyp(\omega)$. Notice that any element of $G_{p,q}$ acting trivially on $\omega$ acts trivially on the whole tree, as $\omega$ contains an edge and every edge in the tree has the same stabiliser. Another way of looking at this subgroup is as the elements that act on $\omega$ by translations.

\begin{lemma}\label{axissubgroup}
For any hyperbolic $h \in G_{p,q}$, and its axis of translation $T^h$, $\Hyp(T^h) = \langle \hat{g} \rangle \times \langle x^p \rangle \cong \mathbb{Z}^2$, where $\hat{g} \in \Stab(T^h)$ is an element acting hyperbolically, and $|l(\hat{g})|$ is minimal among such elements.
\end{lemma}
\begin{proof}
Since translation lengths are non-negative integers, the minimum translation length of an element acting hyperbolically along $T^h$ is well defined and achieved (because any set of non-negative integers has a well defined and realised minimum). So such a $\hat{g}$ exists. Notice it is not necessarily unique.

Take $g \in \Stab(T^h)$. Since $|l(\hat{g})| \leq |l(g)|$, there are integers $m$ and $r$ with $0 \leq r < |l(\hat{g})|$ such that $l(g) = m l(\hat{g}) + r$. Now notice $g\hat{g}^{-m}$ acts along $T^h$ with translation length $r$, so it must be that $r = 0$ by minimality of $|l(\hat{g})|$. So $g\hat{g}^{-m}$ fixes $T^g$ pointwise, in particular fixes an edge, and so $g\hat{g}^{-m} \in \langle x^p \rangle$. As required, $g \in \langle \hat{g} \rangle \times \langle x^p \rangle$.
\end{proof}

This motivates the following definition, which captures the elements $h$ acting hyperbolically in the $\langle \hat{g} \rangle$ factor of $\Hyp(T^h)$.

\begin{definition}\label{directdefn}
We will say $g \in G_{p,q}$ acts \emph{directly} along axis $\omega$ if it acts hyperbolically along $\omega$ and there is $\hat{g}$ acting along $\omega$ with minimal non-zero translation length such that $g = \hat{g}^k$ where $0 \neq k \in \mathbb{Z}$.
\end{definition}

\begin{remark} If $g$ acts hyperbolically along $\omega$, then, for some $k \in \mathbb{Z}$, $g(x^p)^k$ acts directly along $\omega$, since elements in the $\hat{g}$ factor in Lemma \ref{axissubgroup} act directly. In particular, the automorphism $\iota_g$ can be represented by $\iota_{g(x^p)^k}$, because conjugation by central elements is trivial. That is, we can always assume that conjugation by an arbitrary element acting hyperbolically is conjugation by an element acting directly along the same axis. \end{remark}

The following example exhibits two elements of a torus knot group acting hyperbolically along the same axis, where one acts directly and the other does not.

\begin{example}
Consider the torus knot group $G_{4,4}$. Let $\omega$ be the axis in the corresponding Bass-Serre tree $T$ through the edge $e_{\langle x^4 \rangle}$ comprised of all vertices of the form $wv_{\langle x \rangle}$ and $wv_{\langle y \rangle}$, where $w$ is a word in $x^2$ and $y^2$.

Then $h = x^2y^2x^2y^2$ is an element acting directly along this axis, since $h = (x^2y^2)^2$, where $x^2y^2$ has translation length 2, which is minimal as distinct vertices in the same orbit in $T$ are at distance at least 2. The element $hx^4$ also acts hyperbolically along $\omega$ (in fact it acts like $h$, since $x^4$ is in the centre so acts trivially). However $hx^4$ does not act directly. This is because it cannot be written as a proper power, and its translation length is 4, but we have just seen there is an element acting along $\omega$ with translation length 2.
\end{example}

Sometimes we will refer to the maximal cyclic subgroup containing $g \in G_{p,q}$, which we will write $\maxc{g}$. In general this is not well defined, for instance the group relation gives us $x^p = y^q$, but $\langle x \rangle$ and $\langle y \rangle$ are cyclic groups both containing $x^p$, neither of which contains the other. The following lemma justifies this notion for non-central elements. We will take care to only use this notion where it is well defined.

\begin{lemma}\label{maxc}
If $g \in G_{p,q}$ is not central, then $\maxc{g}$ is well defined.
\end{lemma}
\begin{proof}
Suppose $g$ is non-central. Then it does not fix any edges. Hence it is either hyperbolic, acting along an axis, or elliptic, fixing a single point.

Suppose $g$ is elliptic fixing a vertex $v$. Then suppose there is $h \in G_{p,q}$ and $k \in \mathbb{Z}$ such that $h^k = g$. Then $h$ is also elliptic, non-central, and therefore fixes a single vertex which must be $v$. So $h \in G_v$. Therfore any cyclic subgroup containing $g$ is a subgroup of $G_v$, which is cyclic and therefore maximal cyclic, so $\maxc{g} = G_v$.

Suppose $g$ is hyperbolic, so acts along an axis. Then suppose there is $h \in G_{p,q}$ and $k \in \mathbb{Z}$ such that $h^k = g$. Then $h$ must be hyperbolic, and must act along the same axis (as powers of an element act along the same axis). So, by Lemma \ref{axissubgroup}, $h \in \Hyp(T^g) \cong \mathbb{Z}^2$. So any cyclic subgroup containing $g$ is a subgroup of this copy of $\mathbb{Z}^2$. But in $\mathbb{Z}^2$ the notion of a maximal cyclic subgroup containing a non-trivial element is well defined, so $\maxc{g}$ is well defined.
\end{proof}

\subsection{The Automorphism Group of $G_{p,q}$}\label{sectionautgroup}

Recall for a group $G$, we write $\Inn(G)$ for the subgroup of $\Aut(G)$ comprised of the inner automorphisms, that is conjugations $\iota_g: h \mapsto ghg^{-1}$. The outer automorphism group is the quotient of the automorphism group by the inner automorphisms, $\Out(G) := \Aut(G) / \Inn(G)$. We will use $[\varphi]$ to denote the outer automorphism class of an automorphism $\varphi$.

The group $\Out(G_{p,q})$ was determined in \cite{GHMR99}.

\begin{notation}
Throughout this paper we will use $\tau$ and $\sigma$ for the following automorphisms of $G_{p,q}$, where $\sigma$ is an automorphism if and only if $p=q$: \begin{align*}
\tau:~& x \mapsto x^{-1}\\
& y \mapsto y^{-1}\\
\sigma:~& x \mapsto y\\
& y \mapsto x.
\end{align*}
\end{notation}

\begin{theorem}\textbf{\cite[Theorem C]{GHMR99}}
If $p \neq q$, $\Out(G_{p,q}) \cong C_2$, where a representative of the sole non-trivial outer automorphism class is $\tau$.

If $p = q$, $\Out(G_{p,q}) \cong C_2 \times C_2$, where a representative of the generators of each factor are given by $\tau$ and $\sigma$.
\end{theorem}

Now we describe an action of $\Aut(G_{p,q})$ on $T$. The action will be compatible with the action of $G_{p,q}$ on $T$, in the sense of Definition \ref{compatibleaction}. This action, and its compatibility, were originally described in \cite{GHMR99}, and were the main tool used to calculate $\Out(G_{p,q})$.

In order to keep the distinction between the two actions clear, throughout this paper we will write $\cdot$ for the action of $\Aut(G)$, and leave the action of $G$ implicit by writing $gx$ for the action of $g \in G$ on a point $x$.

Note that to define an action on a tree by isometries we only need to give an action on the vertices which respects the distance function.

The action is given by the statement of Lemma \ref{actionofautgpq}, which is an application of Theorem A from \cite{GHMR99}. 

\begin{lemma}\label{actionofautgpq}
$Aut(G_{p,q})$ acts on $T$ by isometries with the following action: for $\psi \in \Aut(G_{p,q})$ and $v$ a vertex, $\psi \cdot v$ is the unique vertex with stabiliser $\psi(G_v)$.
\end{lemma}
\begin{proof}
This is a special case of Theorem A in \cite{GHMR99}. For completeness we check the hypothesis of Theorem A apply, but do not define them here as they will not be used again. The action of $G_{p,q}$ on $T$ is \emph{proper} and \emph{locally-tame} in the sense of that paper, because for each vertex $v$ of $T$ there is one orbit of edges adjacent to $v$ under the action of $G_v$ (this gives local-tameness), and each vertex has at least 2 adjacent edges (since each of these edges is in the same orbit, this gives properness).
\end{proof}

Notice that, although $G_{p,q}$ acts without inversions, the action of $Aut(G_{p,q})$ may well have inversions.

Recall $e_{\langle x^p \rangle}$ is a fundamental domain for the action of $G_{p,q}$ on $T$.

\begin{lemma}
The action of the automorphism $\tau$ fixes $e_{\langle x^p \rangle}$ pointwise.

When $p=q$ and the automorphism $\sigma$ exists, the action of the automorphism $\sigma$ exchanges $v_{\langle x \rangle}$ and $v_{\langle y \rangle}$, inverting the edge $e_{\langle x^p \rangle}$.
\end{lemma}
\begin{proof}
This follows from the description of the action in Lemma \ref{actionofautgpq} and the fact that $\tau$ preserves $\langle x \rangle$ and $\langle y \rangle$, whilst $\sigma$ exchanges them.
\end{proof}

Note that in the case of $\tau$ the fundamental domain needn't be, and in general won't be, the entirety of $T^{\tau}$ (the set of points on the tree fixed by $\tau$).

Intuitively it is instructive to think of $\tau$ as a reflection over the edge that forms the fundamental domain, and $\sigma$ as a reflection across a line bisecting it. Of course this description depends on a particular embedding into 2-dimensional space, but will be useful geometric intuition. Figure \ref{diagramOfAction} depicts this point of view.

\begin{figure}[H]
\centering
\begin{tikzpicture} [%
    nd/.style = {circle,fill=white,text=black,inner sep=0pt},
    tn/.style = {node distance=1pt},
    ed/.style={-,black,fill=none},
    sm/.style={-, dashed,red,fill=none}]

    \node[nd] (x) at (0,0) {\sffamily $\langle x \rangle$};
    \node[nd] (y) [right=of x] {\sffamily $\langle y \rangle$};
    
    \node[nd] (xy) [above left=of x] {\sffamily $x \langle y \rangle$};
    \node[nd] (x^2y) [below left=of x] {\sffamily $x^2 \langle y \rangle$};
    \node[nd] (yx) [above right=of y] {\sffamily $y \langle x \rangle$};
    \node[nd] (y^2x) [below right=of y] {\sffamily $y^2 \langle x \rangle$};
    
    \node[nd] (xyx) [above=of xy] {\sffamily $xy \langle x \rangle$};
    \node[nd] (xy^2x) [left=of xy] {\sffamily $xy^2 \langle x \rangle$};
    \node[nd] (x^2yx) [left=of x^2y] {\sffamily $x^2y \langle x \rangle$};
    \node[nd] (x^2y^2x) [below=of x^2y] {\sffamily $x^2y^2 \langle x \rangle$};
    \node[nd] (yxy) [above=of yx] {\sffamily $yx \langle y \rangle$};
    \node[nd] (yx^2y) [right=of yx] {\sffamily $yx^2 \langle y \rangle$};
    \node[nd] (y^2xy) [right=of y^2x] {\sffamily $y^2x \langle y \rangle$};
    \node[nd] (y^2x^2y) [below=of y^2x] {\sffamily $y^2x^2 \langle y \rangle$};
    
    \node[nd] (taulef) at (-4,0) {\sffamily $\tau$};
    \node[nd] (taurig) at (5.6, 0) {\sffamily $\tau$};
	\node[nd] (sigtop) at (0.8,4) {\sffamily $\sigma$};
    \node[nd] (sigbot) at (0.8, -4) {\sffamily $\sigma$};
    
    \draw[ed] (x) -- (y);
    
    \draw[ed] (x) -- (xy);
    \draw[ed] (x) -- (x^2y);
    \draw[ed] (y) -- (yx);
    \draw[ed] (y) -- (y^2x);
    
    \draw[ed] (xy) -- (xyx);
    \draw[ed] (xy) -- (xy^2x);
    \draw[ed] (x^2y) -- (x^2yx);
    \draw[ed] (x^2y) -- (x^2y^2x);
    \draw[ed] (yx) -- (yxy);
    \draw[ed] (yx) -- (yx^2y);
    \draw[ed] (y^2x) -- (y^2xy);
    \draw[ed] (y^2x) -- (y^2x^2y);
    
    \draw[sm] (taulef) -- (taurig);
	\draw[sm] (sigtop) -- (sigbot);
    
\end{tikzpicture}

\caption{\label{diagramOfAction} A fragment of the Bass-Serre tree for $G_{3,3}$. When the tree is drawn appropriately, the actions of $\tau$ and $\sigma$ can be seen as reflections across the labelled axes.}
\end{figure}
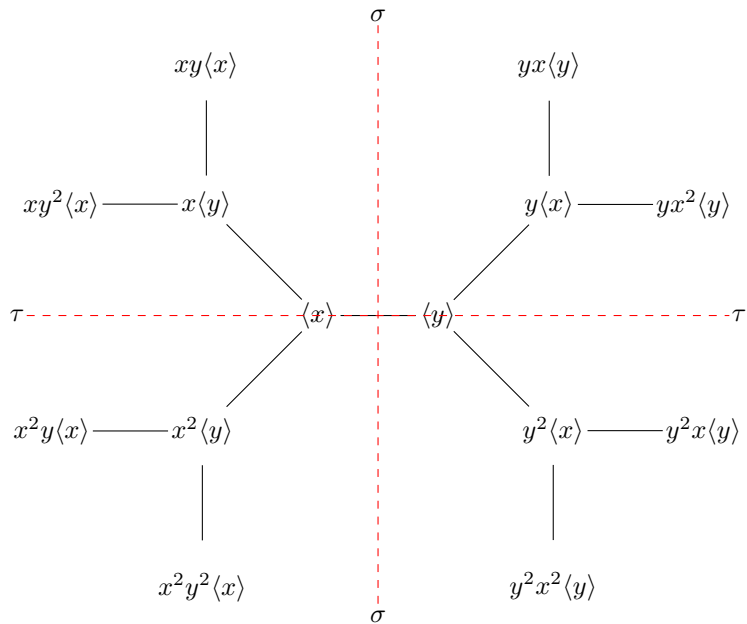

\section{Compatible Actions}\label{compatiblesection}

In this section we discuss the general setting of a group and its automorphism group acting on a set $X$ (at this stage with no further structure), in a compatible way, in the sense of the following definition. The action in the previous section will be a particular case of this setup.

\begin{definition}\label{compatibleaction}
Let $G$ be a group acting on a set $X$ via $\Omega: G \rightarrow \Aut(X)$. We say $\Aut(G)$ has a \emph{compatible} action on $X$ if $\Psi: \Aut(G) \rightarrow \Aut(X)$ is an action of $\Aut(G)$ on $X$ such that $\Psi \circ i = \Omega$, where $i: G \rightarrow \Aut(G)$ is the homomorphism $g \mapsto \iota_g$.
\end{definition}

This should be understood that an action of $\Aut(G)$ on $X$ is compatible with the action of $G$ on $X$ if and only if $\iota_g \cdot x = gx$ for each $x \in X$ and each inner automorphism $\iota_g$ of $G$.

\begin{lemma}\label{compatibilitylemma}
Let $G$ be a group acting on a set $X$ such that $\Aut(G)$ acts on $X$ in a compatible way. Then for all $\varphi \in \Aut(G)$, $g \in G$, and $x \in X$, $$\varphi \cdot gx = \varphi(g) (\varphi \cdot x).$$
\end{lemma}
\begin{proof}
\begin{align*}
\varphi \cdot gx &= \varphi \iota_g \cdot x\\
&= \iota_{\varphi(g)} \varphi \cdot x\\
&= \varphi(g) (\varphi \cdot x).\\
\end{align*}

The first and third equalities follow from compatibility, and the second is because $\varphi \circ \iota_g = \iota_{\varphi(g)} \circ \varphi$ in the automorphism group.
\end{proof}

We will most often use this lemma in the form of the following immediate corollaries.  In this setting, where $X$ is not necessarily a metric space, $$X^\varphi := \{x \in X | \varphi \cdot x = x\},$$ is the set of points fixed by $\varphi$.

\begin{corollary}\label{commutingactions}
If $G$ and $X$ are as above, then for all $\varphi \in \Aut(G)$, $g \in \Fix(\varphi)$, and $x \in X$,  $\varphi \cdot gx = g(\varphi \cdot x)$.
\end{corollary}

\begin{corollary}\label{phiaction}
If $G$ and $X$ are as above, then for all $\varphi \in \Aut(G)$, $g \in G$, and $x \in X^{\varphi}$, $\varphi \cdot gx = \varphi(g)x$.
\end{corollary}

The next corollary connects compatible actions to fixed points of automorphisms.

\begin{corollary}\label{fixaction}(See also \cite[Proposition 3.9]{S02})
If $G$ and $X$ are as above, then for all $\varphi \in \Aut(G)$, $\Fix(\varphi) \subseteq \Stab(X^\varphi)$. 
\end{corollary}
\begin{proof}
Let $g \in \Fix(\varphi)$. Then Corollary \ref{commutingactions} and Corollary \ref{phiaction} show that for $x \in X^\varphi$, $\varphi \cdot gx = gx$, so $gx \in X^\varphi$ also. Therefore $g \in \Stab(X^\varphi)$.
\end{proof}

If $X$ is a metric space with $G$ and $\Aut(G)$ acting by isometries, we can generalise our definition of $X^\varphi$ to the points where $\varphi$ achieves its translation length, as in Section \ref{isosoftrees}. That is $X^\varphi := \{x \in X | d(\varphi \cdot x, x) = l(\varphi)\}$. Corollary \ref{transaction} adapts Corollary \ref{fixaction} to this setting, where it still holds.

\begin{corollary}\label{transaction}
Suppose $G$ and $X$ are as above, and $X$ is a metric space with $G$ and $\Aut(G)$ acting by isometries. Then for all $\varphi \in \Aut(G)$, $\Fix(\varphi) \subseteq \Stab(X^\varphi)$.
\end{corollary}
\begin{proof}
For arbitrary $g \in \Fix(\varphi)$, notice that, $$d_X(gx, \varphi \cdot gx) = d_X(gx, g(\varphi \cdot x)) = d_X(x, \varphi \cdot x),$$ where the first equality is by Corollary \ref{commutingactions} and the second is since $g$ acts by isometries. Therefore $g \in \Stab(X^\varphi)$.
\end{proof}

\section{Fixed Points of Inner Automorphisms}

We return to the groups $G_{p,q} = \langle x, y | x^p = y^q \rangle$. The results in this section are well known, but we provide self-contained proofs using the compatible action.

Notice that understanding the fixed points of $\iota_g$ amounts to understanding $C(g)$, the centralizer of g, because $\iota_g(h) = h$ if and only if $gh = hg$. Recall that $\maxc{g}$ is the maximal cyclic subgroup containing $g$, and for non-central elements this is well defined by Lemma \ref{maxc}.

\begin{lemma}\label{ellipticcentralizers}
Let $g \in G_{p,q}$ be elliptic and not in the centre. Then $\Fix(\iota_g) = C(g) = \maxc{g} \cong \mathbb{Z}$.
\end{lemma}
\begin{proof}
Since, by Lemma \ref{centre}, the centre is the stabiliser of each edge, $g$ fixes no edges. So $g$ fixes a single vertex, call it $v$. By compatibility, $\iota_g$ acts like $g$ did. By Corollary \ref{fixaction}, $\Fix(\iota_g) \leq \Stab(v) = \maxc{g}$. 

Conversely $g \in \maxc{g}$, and $\maxc{g}$ is abelian so $\maxc{g} \leq C(g) = \Fix(\iota_g)$. This completes the proof.
\end{proof}

To handle hyperbolic elements, we make use of the remark following Definition \ref{directdefn}, and only consider conjugation by elements which act directly. Such a representative exists for any inner automorphism corresponding to a hyperbolic element.

\begin{lemma}\label{hyperboliccentralizers}
Let $g \in G_{p,q}$ be hyperbolic. Further assume $g$ acts directly along its axis. Then $\Fix(\iota_g) = C(g) = \maxc{g} \times \langle x^p \rangle \cong \mathbb{Z}^2$.
\end{lemma}
\begin{proof}
By compatibility $\iota_g$ acts like $g$, along the axis $T^g$. By Corollary \ref{transaction}, $\Fix(\iota_g) \leq \Stab(T^g)$.

Notice that no non-central elliptic element can be in $C(g)$, because if $g$ commuted with such an element it would be in that element's centraliser, but the centralisers of non-central elliptic elements contain only elliptic elements by Lemma \ref{ellipticcentralizers}.

Therefore $\Fix(\iota_g) \leq \Hyp(T^g) = \maxc{g} \times \langle x^p \rangle$, where the equality is by Lemma \ref{axissubgroup}, and the fact that $g$ acts directly. Since $\maxc{g}$ is an abelian subgroup containing $g$ and $\langle x^p \rangle$ is the centre, all elements in $\maxc{g} \times \langle x^p \rangle$ commute with $g$ and $\Fix(\iota_g) \geq \maxc{g} \times \langle x^p \rangle$, completing the proof.
\end{proof}

\section{Fixed Points of Automorphisms in $[\tau]$}

Recall $\tau: x \mapsto x^{-1}, y \mapsto y^{-1}$ is a non-inner automorphism. We now consider fixed points of automorphisms in the class $[\tau] \in \Out(G_{p,q})$. In the setting $p \neq q$ this is the only non-trivial outer automorphism class, so the results in this section will complete our understanding of fixed points in these cases.

An automorphism in this class can be written $\tau_g := \iota_g \circ \tau$. Notice $\tau$ is order 2 (in the automorphism group, not just the outer automorphism group).

We will split into two cases, based on whether $\tau_g$ is finite or infinite order, using the following characterisation.

\begin{lemma}\label{finiteordergtaug}
Take $\tau_g = \iota_g \circ \tau$ an automorphism as above. The following are equivalent:

\begin{enumerate}
\item $g\tau(g) = 1$,
\item $g\tau(g) \in Z(G_{p,q})$,
\item $\tau_g$ has order 2,
\item $\tau_g$ has finite order.
\end{enumerate}
\end{lemma}
\begin{proof}
$(1 \implies 2)$. If $g\tau(g) = 1$, then clearly $g\tau(g) \in Z(G_{p,q})$.

$(2 \implies 3)$. Suppose $g\tau(g) \in Z(G_{p,q})$. Then $\tau_g^2 = \iota_g \tau \iota_g \tau = \iota_{g\tau(g)} \tau^2 = id$, where the final equality is because $\tau$ is order 2 and conjugation by a central element is trivial. So $\tau_g$ has order 2.

$(3 \implies 4)$. If $\tau_g$ has order 2, then clearly it has finite order.

$(4 \implies 1)$. Finally suppose $\tau_g$ has finite order. Then so does $\tau_g^2 = \iota_{g\tau(g)}$. Say it has order $k$. Since $\iota_{g\tau(g)}^k = \iota_{(g\tau(g)^k}$ and conjugation by an element is the identity if and only if that element is in the centre, this means $(g\tau(g))^k \in \langle x^p \rangle$. However, consider the homomorphism $ht: G_{p,q} \rightarrow \mathbb{Z}$ given by $x \mapsto q$ and $y \mapsto p$. Notice $ht((g\tau(g))^k) = k ~ ht(g\tau(g)) = k(ht(g) - ht(g)) = 0$. On the other hand $ht(x^p) = qp \neq 0$, so $1$ is the only element of the centre in the kernel of $ht$. So $(g\tau(g))^k = 1$, but $G_{p,q}$ is torsion free so $g\tau(g) = 1$ as required.
\end{proof}

If $\tau_g$ is of infinite order we will calculate the fixed points algebraically by the following lemma. The finite order automorphisms will be dealt with geometrically.

\begin{lemma}\label{algebraictau}
If $\tau_g$ is of infinite order, then $\Fix(\tau_g)= \maxc{g\tau(g)} \cong \mathbb{Z}.$
\end{lemma}
\begin{proof}
Notice first that, by Lemma \ref{finiteordergtaug}, $g\tau(g) \notin Z(G_{p,q})$. Therefore $g\tau(g)$ is non-trivial; by Lemma \ref{maxc} the maximal cyclic subgroup containing $g\tau(g)$, $\maxc{g\tau(g)}$, is well defined; and for $k \in \mathbb{Z}$, if $g\tau(g)$ has a $k$th root, that is $h$ such that $h^k = g\tau(g)$, then it is unique. For the final claim notice that any $k$th root is in $\maxc{g\tau(g)}$, which is infinite cyclic so has unique $k$th roots.

First we show $\maxc{g\tau(g)} \leq \Fix(\tau_g)$. Notice that $$\tau_g(g \tau(g)) = \iota_g \circ \tau (g \tau(g)) = g \tau(g) g g^{-1} = g \tau (g).$$ Since $g\tau(g)$ is a fixed point of $\tau_g$ then all of its powers certainly are as well. Furthermore, if $h^k = g\tau(g)$ for some $k \in \mathbb{Z}$ then $g\tau(g) = \tau_g(h^k) = \tau_g(h)^k$, but by uniqueness of $k$th roots, $\tau_g(h) = h$. This completes the proof of this inclusion.

It remains to show $\Fix(\tau_g) \leq \maxc{g\tau(g)}$. A fixed point of an automorphism is always a fixed point of any power of that automorphism, so $\Fix(\tau_g) \leq \Fix(\tau_g^2)$. Now note $\tau_g^2 = \iota_{g\tau(g)}$, and in particular is an inner automorphism. Since no non-zero power of $g\tau(g)$ is central (if it were, $\tau_g$ would have finite order), it must be that $g\tau(g)$ acts hyperbolically. Thus using our understanding of centralizers from Lemma \ref{hyperboliccentralizers}, we see $\Fix(\tau_g^2) = \langle \hat{g} \rangle \times \langle x^p \rangle \cong \mathbb{Z}^2$, where $\hat{g}$ acts directly and is not a proper power. Notice that we cannot assume $g\tau(g)$ itself acted directly.

Since $\tau_g(x^p) = gx^{-p}g^{-1} = x^{-p}$, we see $\tau_g$ can't fix anything in the centre. So $\Fix(\tau_g)$ must be a subgroup of rank at most 1 in $\Fix(\tau_g^2)$, that is trivial or cyclic. By the above $\maxc{g\tau(g)} \leq \Fix(\tau_g)$, but since $\maxc{g\tau(g)}$ is a maximal cyclic subgroup, $\Fix(\tau_g) = \maxc{g\tau(g)}$ as required.

\end{proof}

We are only left to address $\tau_g$ of finite order. Our approach will be to establish $T^{\tau_g}$ (that is the subtree fixed pointwise by the action of $\tau_g$) in these cases, and then apply Corollary \ref{fixaction} to restrict our attention to the setwise stabiliser $\Stab(T^{\tau_g})$.

\begin{lemma}\label{trueformlemma}
If $\tau_g$ is of finite order, then there exists $h \in G$ and $r \in \mathbb{Z}$ such that $g = h x^r \tau(h)^{-1}$ or $g = h y^r \tau(h)^{-1}$. In particular, $\tau_g = \iota_h \iota_{x^r} \tau \iota_h^{-1}$ or $\tau_g = \iota_h \iota_{y^r} \tau \iota_h^{-1}$.
\end{lemma}
\begin{proof}
Notice that $\tau_g \cdot v_{\langle x \rangle} = \iota_g \tau \cdot v_{\langle x \rangle} = g (\tau \cdot v_{\langle x \rangle}) = gv_{\langle x \rangle}$. Furthermore $\tau_g \cdot gv_{\langle x \rangle} = g \tau(g) v_{\langle x \rangle} = v_{\langle x \rangle}$, where the first equality is using Corollary \ref{phiaction}, and the last equality holds because $g\tau(g) = 1$ by Lemma \ref{finiteordergtaug}. Since $\tau_g$ interchanges $v_{\langle x \rangle}$ and $g v_{\langle x \rangle}$, it fixes the midpoint of the geodesic between them. The length of the path is even (because the tree is alternating between vertices corresponding to cosets of $\langle x \rangle$ and $\langle y \rangle$), so this midpoint is a vertex, without loss of generality suppose it corresponds to a coset of $\langle y \rangle$ and call it $hv_{\langle y \rangle}$.

We compute $h v_{\langle y \rangle} = \tau_g h v_{\langle y \rangle} = g \tau(h) v_{\langle y \rangle}$, where the first equality is by assumption the vertex is fixed, and the second uses compatibility and Lemma \ref{phiaction}. So $h^{-1}g\tau(h) \in G_{v_{\langle y \rangle}} = \langle y \rangle$. Rearranging we see $g = hy^r\tau(h)^{-1}$, for some $r \in \mathbb{Z}$, as required. If $g$ is of this form then $\tau_g = \iota_g \tau = \iota_h \iota_{y^r} \iota_{\tau(h)^{-1}} \tau = \iota_h \iota_{y^r} \tau \iota_h^{-1}$.

Likewise, if the midpoint of the geodesic considered about was $hv_{\langle x \rangle}$, then, for some $r$, $\tau_g = \iota_h \iota_{x^r} \tau \iota_h^{-1}$.
\end{proof}

With Lemma \ref{trueformlemma} in mind, we will first consider automorphisms of the form $\tau_{x^r}$ and $\tau_{y^r}$, for $r \in \mathbb{Z}$, because generic $\tau_g$ of finite order are conjugate to these automorphisms by inner automorphisms, which act like group elements by compatibility. Thus $T^{\tau_g}$ will just be translates of $T^{\tau_{x^r}}$ and $T^{\tau_{y^r}}$.

The following lemmas are the key ingredients in determining $T^{\tau_{x^r}}$, which will be done in Proposition \ref{xtaustabset}. The proof will be by induction with the Lemma \ref{stabfundamentaldomain} being the base case and Lemma \ref{stabsetextensionlemma} being the inductive case. In the following, it will be useful to have in mind the informal point of view in Figure \ref{diagramOfAction} (Section \ref{sectionautgroup}), that $\tau$ is a reflection over an axis extending the edge $e_{\langle x^p \rangle}$.

\begin{lemma}\label{stabfundamentaldomain}
For $r \in \mathbb{Z}$ the automorphism $\tau_{x^r}$ fixes $v_{\langle x \rangle}$. Furthermore, for $k \in \mathbb{Z}$, $\tau_{x^r}$ fixes the edge $x^k e_{\langle x^p \rangle}$ if and only if $r = 2k$ (mod $p$).
\end{lemma}
\begin{proof}
Clearly $\tau_{x^r} = \iota_{x^r} \tau$ fixes $v_{\langle x \rangle}$, as each $\tau$ and $\iota_{x^r}$ preserve $\langle x \rangle$ as a set. Notice this is one endpoint of $x^k e_{\langle x^p \rangle}$.

Since $\tau$ fixes $v_{\langle y \rangle}$, by Corollary \ref{phiaction}, $\tau_{x^r} \cdot x^{k} v_{\langle y \rangle} = \iota_{x^r} \tau \cdot x^{k} v_{\langle y \rangle} = x^{r-k} v_{\langle y \rangle}$. So the edge $x^k e_{\langle x^p \rangle}$ is fixed by $\iota_{x^r} \tau$ if and only if $x^{k} v_{\langle y \rangle} = x^{r-k} v_{\langle y \rangle}$. Since $x^p = y^q$ is the minimal positive power of $x$ in $\langle y \rangle$, this happens if and only if $r-k = k$ (mod $p$), that is $r = 2k$ (mod $p$).
\end{proof}

Lemma \ref{stabsetextensionlemma}, which will be used inductively in Proposition \ref{xtaustabset}, takes an edge $x^{k}he_{\langle x^p \rangle}$, assumed to be in $T^{\tau_{x^r}}$, and then characterises which of the edges sharing an endpoint with $x^{k}he_{\langle x^p \rangle}$ are also in $T^{\tau_{x^r}}$.

\begin{lemma}\label{stabsetextensionlemma}
Let $r = 2k$ (mod $p$) and suppose $x^{k}he_{\langle x^p \rangle}$ is an edge fixed by $\tau_{x^r}$, for some $h\in G_{p,q}$.  Then, for each $i \in \mathbb{Z}$, the edge $x^{k}hy^ie_{\langle x^p \rangle}$ is fixed by $\tau_{x^r}$ if and only if $i$ is an integer multiple of $\frac{q}{2}$.

Likewise the edge $x^khx^ie_{\langle x^p \rangle}$ is fixed by $\tau_{x^r}$ if and only if $i$ is an integer multiple of $\frac{p}{2}$.
\end{lemma}
\begin{proof}
One endpoint of the edge $x^{k}hy^ie_{\langle x^p \rangle}$ is the vertex $x^{k}hy^iv_{\langle y \rangle} = x^{k}hv_{\langle y \rangle}$, which is in $x^{k}he_{\langle x^p \rangle}$ so is fixed by $\tau_{x^r}$ by assumption.

So $x^{k}hy^ie_{\langle x^p \rangle}$ is fixed if and only if its other endpoint, $x^{k}hy^iv_{\langle x \rangle}$, is.

We claim that $h^{-1}\tau(h) \in Z(G_{p,q})$. To this end we calculate, \begin{align*}
x^{k}he_{\langle x^p \rangle} &= \tau_{x^r} \cdot x^{k}he_{\langle x^p \rangle}\\
&= \iota_{x^r} \tau \cdot x^{k}he_{\langle x^p \rangle}\\
&= \iota_{x^r} \cdot x^{-k}\tau(h)e_{\langle x^p \rangle}\\
&= x^{r-k} \tau(h) e_{\langle x^p \rangle}\\
&= x^{k+ap} \tau(h) e_{\langle x^p \rangle}\\
&= x^k\tau(h) e_{\langle x^p \rangle},\\
\end{align*}
where the first equality is by assumption, and using Corollary \ref{phiaction} for the third equality, the fourth equality is by compatibility, the fifth holds for some $a$ since $r = 2k$ (mod $p$), and the final equality is because $x^{ap}$ is in the centre and edge stabilizer. This means $h^{-1}\tau(h)$ is in the stabilizer of $e_{\langle x^p \rangle}$, which is $Z(G_{p,q})$.

So considering the other endpoint of $x^{k}hy^ie_{\langle x^p \rangle}$, that is the vertex $x^{k}hy^iv_{\langle x \rangle}$, we see, \begin{align*}
\tau_{x^r} \cdot x^{k}hy^iv_{\langle x \rangle} &= x^r x^{-k} \tau(h) y^{-i} v_{\langle x \rangle}\\
&= x^{k+ap} h y^{-i} v_{\langle x \rangle}\\
&= x^{k} h y^{-i} v_{\langle x \rangle},\\
\end{align*} where the first equality is by Corollary \ref{phiaction}, the second is because $r=2k$ (mod $p$), and the third is since $x^p$ is in the centre. The vertex $x^{k}hy^iv_{\langle x \rangle}$, and thus the edge $x^{k}hy^ie_{\langle x^p \rangle}$, is fixed if and only if $(y^{-i})^{-1}y^{i} = y^{2i}  \in \langle x \rangle$. Since $\langle y \rangle \cap \langle x \rangle = \langle y^q \rangle$, this occurs if and only if $y^{2i}$ is a power of $y^q = x^p$, which occurs if and only if $i$ is an integer multiple of $\frac{q}{2}$.

The second part of the conclusion follows by symmetrical argument.
\end{proof}

With the previous two lemmas, we are now ready to determine the subtree $T^{\tau_{x^r}}$ exactly.

\begin{proposition}\label{xtaustabset}

The subtree $T^{\tau_{x^r}}$ fixed pointwise by $\tau_{x^r}$ is as follows.

\begin{enumerate}

\item Suppose there is $k \in \mathbb{Z}$ such that $r = 2k$ (mod $p$). 
\begin{enumerate}
\item If $p$ and $q$ are both odd then $T^{\tau_{x^r}}$ is an edge.
\item If exactly one of $p$ and $q$ is odd then $T^{\tau_{x^r}}$ is 2 edges.
\item If $p$ and $q$ are both even then $T^{\tau_{x^r}}$ is a bi-infinite line.
\end{enumerate} 

\item If there is no $k \in \mathbb{Z}$ such that $r = 2k$ (mod $p$), then $T^{\tau_{x^r}} = v_{\langle x \rangle}$ is a single point.

\end{enumerate}
\end{proposition}

\begin{proof}
If there is no $k \in \mathbb{Z}$ such that $r = 2k$ (mod $p$), then the result is given by Lemma \ref{stabfundamentaldomain} since $v_{\langle x \rangle}$ is fixed and none of the neighbouring edges are, and $T^{\tau_{x^r}}$ is connected by Proposition \ref{geometricfixsubtree}.

So suppose $r=2k$ (mod $p$). By Lemma \ref{stabfundamentaldomain}, $x^{k} e_{\langle x^p \rangle}$ is an edge in the tree $T^{\tau_{x^r}}$. The result will follow by inductive application of Lemma \ref{stabsetextensionlemma} to build out from this starting point.

Notice that, by Proposition \ref{geometricfixsubtree}, $T^{\iota_{x^r}\tau}$ must be a connected subtree, so given $T' \subseteq T^{\iota_{x^r}\tau}$, where $T'$ is a non-empty union of edges including their endpoints, either there is an edge in $T^{\tau_{x^r}} \setminus T'$ sharing a vertex with an edge in $T'$, or $T' = T^{\tau_{x^r}}$. An edge shares a vertex with an arbitrary edge $he_{\langle x^p \rangle}$ if and only if it is of the form $hx^ie_{\langle x^p \rangle}$ or $hy^ie_{\langle x^p \rangle}$, for some integer $i$.

Lemma \ref{stabsetextensionlemma} dictates that a vertex of $T^{\tau_{x^r}}$ can be contained in at most 2 edges of $T^{\tau_{x^r}}$ (because $x^khx^{a\frac{i}{2}}e_{\langle x^p \rangle} = x^khe_{\langle x^p \rangle}$ if $a$ is even). So $T^{\tau_{x^r}}$ is a (finite or bi-infinite) line. Moreover, also by Lemma \ref{stabsetextensionlemma}, a vertex labelled by a coset of $x$ (respectively $y$) in $T^{\tau_{x^r}}$ is connected to 2 edges if and only if $p$ (respectively $q$) is even, and 1 edge otherwise.

We now have everything we need to conclude.

If $p$ and $q$ are both odd the edge $x^{k} e_{\langle x^p \rangle}$ is all of $T^{\tau_{x^r}}$.

If $p$ is even and $q$ is odd, then the second edge in $T^{\tau_{x^r}}$ is $x^{\frac{r}{2}} x^{\frac{p}{2}} e_{\langle x^p \rangle}$. 

If $q$ is even and $p$ is odd, then the second edge in $T^{\tau_{x^r}}$ is $x^{\frac{r}{2}} y^{\frac{q}{2}} e_{\langle x^p \rangle}$.

If both $p$ and $q$ are even, then $T^{\tau_{x^r}}$ is the subtree of edges $x^{\frac{r}{2}}w e_{\langle x^p \rangle}$, where $w$ is any word in $x^{\frac{p}{2}}$ and $y^{\frac{q}{2}}$.
\end{proof}

\begin{remark}
Proposition \ref{xtaustabset} also applies to $T^{\tau_{y^r}}$, by putting $y^r$ into the role of $x^r$ in each of the preceding lemmas, where the conditions are now on whether $r = 2k$ modulo $q$ for some $k$.
\end{remark}

\begin{proposition}\label{stabset}
Let $\tau_g$ be of finite order. By Lemma \ref{trueformlemma}, there exists $h \in G_{p,q}$ and $r \in \mathbb{Z}$ such that $g = hx^r\tau(h)^{-1}$ or $g = hy^r\tau(h)^{-1}$. Assume the former.

\begin{enumerate}

\item Suppose there is $k \in \mathbb{Z}$ such that $r = 2k$ (mod $p$). 
\begin{enumerate}
\item If $p$ and $q$ are both odd then $T^{\tau_g}$ is an edge.
\item If exactly one of $p$ and $q$ is odd then $T^{\tau_g}$ is 2 edges.
\item If $p$ and $q$ are both even then $T^{\tau_g}$ is a bi-infinite line.
\end{enumerate} 

\item If there is no $k \in \mathbb{Z}$ such that $r = 2k$ (mod $p$), then $T^{\tau_g}$ is a single point.

\end{enumerate}

The case where $g = hy^r\tau(h)^{-1}$ is symmetrical, conditioned on whether $r$ is a multiple of 2 modulo $q$ instead.
\end{proposition}
\begin{proof}
We prove the stated case, where $g = hx^r\tau(h)^{-1}$. By compatibility, $\iota_h$ acts like $h$, so the action of $\tau_g$ is the action of $\tau_{x^r}$ translated by $h$. This means $T^{\tau_g}$ is a translation of $T^{\iota_{x^r}\tau}$, and the proposition follows by Proposition \ref{xtaustabset}.
\end{proof}

Since, by Corollary \ref{fixaction}, $\Fix(\tau_g) \leq G$ must stabilise $T^{\tau_g}$ as a set, this will allow us to find the fixed points of $\tau_g$.

\begin{lemma}\label{ellipticfixedpoints}
Suppose $\tau_g$ is of finite order and at least one of $p$ and $q$ are odd. Then $\Fix(\tau_g) = \{1\}$.
\end{lemma}
\begin{proof}
By Proposition \ref{stabset}, in these cases $T^{\tau_g}$ is just a point or 1 or 2 edges, so any element that fixes this subtree (setwise) must be elliptic, and hence of the form $hx^kh^{-1}$ or $hy^kh^{-1}$. Since $\tau_g(hx^kh^{-1}) = g\tau(h)x^{-k}(g\tau(h))^{-1}$ and $x^k$ is not conjugate to $x^{-k}$ for $k \neq 0$, no conjugate of a non-zero power of $x$ can be a fixed point. Likewise for conjugates of non-zero powers of $y$. The lemma follows since no non-trivial elliptic elements can be fixed points..
\end{proof}

If $p$ and $q$ are both even, $T^{\tau_g}$ may be a bi-infinite line. This means there are hyperbolic elements which are candidate elements of $\Fix(\tau_g)$. Note in the following we no longer need to check divisibility of $r$ by 2 modulo $p$ (or $q$), because $p$ and $q$ are even so this is equivalent to $r$ being even.

\begin{lemma}\label{hyperbolicfixedpoints}
Suppose $\tau_g$ is of finite order and that $p$ and $q$ are both even. We condition on the structure of the form in Lemma \ref{trueformlemma}.

If $\tau_g = \iota_h \iota_{x^r} \tau \iota_h^{-1}$, then $\Fix(\tau_g) = \{1\}$ if $r$ is odd, and $$Fix(\tau_g) = \langle hx^{\frac{r}{2}}x^{\frac{p}{2}}y^{-\frac{q}{2}}x^{-\frac{r}{2}}h^{-1} \rangle \cong \mathbb{Z},$$ if $r$ is even.

Likewise if $\tau_g = \iota_h \iota_{y^r} \tau \iota_h^{-1}$, then $\Fix(\tau_g) = \{1\}$ if $r$ is odd, and $$Fix(\tau_g)= \langle hy^{\frac{r}{2}}x^{\frac{p}{2}}y^{-\frac{q}{2}}y^{-\frac{r}{2}}h^{-1} \rangle \cong \mathbb{Z},$$ if $r$ is even.
\end{lemma}
\begin{proof}
We consider the case where $\tau_g = \iota_h \iota_{x^r} \tau \iota_h^{-1}$.

If $r$ is odd, then $T^{\tau_g}$ is just a point, by Proposition \ref{stabset}. By Corollary \ref{fixaction}, any candidate fixed point needs to be elliptic. For the reasons given in Lemma \ref{ellipticfixedpoints}, $\varphi_g$ can have no non-trivial elliptic fixed points.

Suppose $r$ is even. By Corollary \ref{fixaction}, $\Fix(\tau_g) \leq \Stab(T^{\tau_g})$. Since elliptic elements can't be fixed points, for the reasons given in Lemma \ref{ellipticfixedpoints}, we can restrict our attention to $\Hyp(T^{\tau_g})$, that is the orientation preserving elements.

By following the proofs of Proposition \ref{xtaustabset} and Proposition \ref{stabset}, one can calculate $T^{\tau_g}$ exactly, and choose a hyperbolic element acting along that axis as a candidate fixed point. A candidate fixed point should be in the kernel of the height homomorphism $ht: G_{p,q} \rightarrow \mathbb{Z}$, given by $x \mapsto q, y \mapsto p$, since $ht(\tau_g(h)) = -ht(h)$ for all $h \in G_{p,q}$. Such an element is $hx^{\frac{r}{2}}x^{\frac{p}{2}}y^{-\frac{q}{2}}x^{-\frac{r}{2}}h^{-1}$.

We find $hx^{\frac{r}{2}}x^{\frac{p}{2}}y^{-\frac{q}{2}}x^{-\frac{r}{2}}h^{-1}$ is in fact a fixed point of $\tau_g$ since, \begin{align*}
\tau_g (hx^{\frac{r}{2}}x^{\frac{p}{2}}y^{-\frac{q}{2}}x^{-\frac{r}{2}}h^{-1}) &= \iota_h \iota_{x^r} \tau \iota_h^{-1} (hx^{\frac{r}{2}}x^{\frac{p}{2}}y^{-\frac{q}{2}}x^{-\frac{r}{2}}h^{-1})\\
&=  \iota_h \iota_{x^{\frac{r}{2}}} (x^{-\frac{p}{2}}y^{\frac{q}{2}})\\
&=  \iota_h \iota_{x^{\frac{r}{2}}} (x^{-\frac{p}{2}}(x^py^{-q})y^{\frac{q}{2}})\\
&=  \iota_h \iota_{x^{\frac{r}{2}}} (x^{\frac{p}{2}}y^{-\frac{q}{2}})\\
&= hx^{\frac{r}{2}}x^{\frac{p}{2}}y^{-\frac{q}{2}}x^{-\frac{r}{2}}h^{-1},
\end{align*} where in the third equality we use the fact $x^p = y^q$. So, $$H := \langle hx^{\frac{r}{2}}x^{\frac{p}{2}}y^{-\frac{q}{2}}x^{-\frac{r}{2}}h^{-1} \rangle \leq \Fix(\tau_g).$$

By Lemma \ref{axissubgroup}, $\Hyp(T^{\tau_g}) \cong \mathbb{Z}^2$ where the factors are $\langle x^p \rangle$ and $H$; notice that the given generator of $H$ has translation length of 2, which is the minimal possible translation length of any element acting hyperbolically, therefore elements of $H$ act directly (see Definition \ref{directdefn}). No non-trivial element of the centre can be fixed by $\tau_g$ as they all act elliptically, therefore $\Fix(\tau_g) = H \cong \mathbb{Z}$.

The proof of the second case is symmetrical.
\end{proof}

\section{Fixed Points of Automorphisms in $[\sigma]$ and $[\tau\sigma]$}

Suppose $p = q$. Then $\Out(G_{p,q}) \cong C_2 \times C_2$. One factor is generated by $[\tau]$, as handled by the previous section. The other is generated by $[\sigma]$, where $\sigma(x) = y$ and $\sigma(y) = x$. All of the non-inner automorphisms in this case not in $[\tau]$ can be written as $\varphi_{g,r} := \iota_g \sigma \tau^{r}$ for $r \in \{0,1\}$. We take the same approach as in the previous section, dealing with infinite order automorphisms algebraically in Lemma \ref{algebraicsigma} and the finite order automorphisms using a geometric argument in Lemma \ref{geometricsigma}. We will use the following characterisation.

\begin{lemma}\label{finiteordergsigmag}
Take $\varphi_{g,r} = \iota_g \sigma \tau^{r}$ an automorphism as above. The following are equivalent:

\begin{enumerate}
\item $g(\sigma\tau^r(g)) \in Z(G_{p,p})$,
\item $\varphi_{g,r}$ has order 2,
\item $\varphi_{g,r}$ has finite order.
\end{enumerate}
\end{lemma}
\begin{proof}
$(1 \implies 2)$. Suppose $g(\sigma\tau^r(g)) \in Z(G_{p,p})$. Then $\varphi_{g,r}^2 = \iota_g\sigma\tau^r\iota_g\sigma\tau^r = \iota_{g(\sigma\tau^r(g))}(\sigma\tau^r)^2 = id$, so $\varphi_{g,r}$ is order 2.

$(2 \implies 3)$. If $\varphi_{g,r}$ is order 2 then clearly it is finite order.

$(3 \implies 1)$. Suppose $\varphi_{g,r}$ is of finite order. Then we consider its action on the Bass-Serre tree. Given it is acting with finite order, $\varphi_{g,r}$ must fix a point on the tree. On the other hand, $\varphi_{g,r}$ cannot fix a vertex of the form $h v_{\langle x \rangle}$, since \begin{align*}
\varphi_{g,r} \cdot h v_{\langle x \rangle} &= \iota_g \sigma \tau^r \cdot h v_{\langle x \rangle}\\
&= \iota_g \sigma \cdot \tau^r(h) v_{\langle x \rangle}\\
&= g\sigma\tau^r(h) v_{\langle y \rangle},
\end{align*} where the second and third equalities are by Lemma \ref{compatibilitylemma}. This vertex cannot be fixed because the orbits of $v_{\langle x \rangle}$ and $v_{\langle y \rangle}$ are distinct. Likewise for vertices of the form $h v_{\langle y \rangle}$. So the point on the tree fixed by $\varphi_{g,r}$ must be the midpoint of an edge which is inverted by the automorphism's action.

The automorphism $\varphi_{g,r}^2 = \iota_{g(\sigma\tau^r(g))}$ also fixes this edge midpoint. By compatibility, $g(\sigma\tau^r(g))$ fixes the edge midpoint and, since the group acts without inversions, the whole edge. But, by Lemma \ref{centre}, each edge is stabilised by the centre, so $g(\sigma\tau^r(g)) \in Z(G_{p,p})$ as required.

\end{proof}

\begin{lemma}\label{algebraicsigma}
Suppose $\varphi_{g,0} = \iota_g \sigma$ is of infinite order. Then $\Fix(\varphi_{g,0}) = C(g(\sigma\tau(g))) \cong \mathbb{Z}^2$.

Suppose $\varphi_{g,1} = \iota_g \sigma \tau$ is of infinite order. Then $\Fix(\varphi_{g,1}) = \maxc{g(\sigma\tau(g))} \cong \mathbb{Z}$.
\end{lemma}

\begin{proof}
The proof for automorphisms $\varphi_{g,1}$ is completely analogous to that of Lemma \ref{algebraictau}.

For automorphisms $\varphi_{g,0}$, the only difference is that the height argument that rules the centre out of the fixed group in the other cases does not apply due to the absence of $\tau$. Indeed, $$\iota_g \sigma (x^p) = gy^pg^{-1} = y^p = x^p,$$ so the centre is fixed by such automorphisms. Thus, $$\maxc{g\sigma(g)} \times \langle x^p \rangle \leq \Fix(\varphi_{g,0}) \leq C(g(\sigma\tau(g))).$$ In particular notice $\Fix(\varphi_{g,0}) \cong \mathbb{Z}^2$, as $\maxc{g\sigma(g)} \times \langle x^p \rangle$ is a rank 2 subgroup of $C(g(\sigma\tau(g)))$ which is isomorphic to $\mathbb{Z}^2$ by Lemma \ref{hyperboliccentralizers}. Recall if $g(\sigma\tau(g))$ acts directly, $\maxc{g\sigma(g)} \times \langle x^p \rangle = C(g(\sigma\tau(g)))$, but in general the former is contained in the latter.

Furthermore since $g(\sigma\tau(g))$ is fixed, it must be that $\varphi_{g,0}(C(g(\sigma\tau(g)))) = C(g(\sigma\tau(g)))$, because some element $h$ commutes with $g(\sigma\tau(g))$ if and only if $\varphi_{g,0}(h)$ commutes with $\varphi_{g,0}(g(\sigma\tau(g))) = g(\sigma\tau(g))$. So $\varphi_{g,0}$ descends to an automorphism of $C(g(\sigma\tau(g)))$. But since $\maxc{g\sigma(g)} \times \langle x^p \rangle \leq \Fix(\varphi_{g,0})$, the only way to extend this to an automorphism of $C(g(\sigma\tau(g)))$ is as the identity, so $C(g(\sigma\tau(g))) \leq \Fix(\varphi_{g,0})$.
\end{proof}

\begin{lemma}\label{geometricsigma}
Suppose $\varphi_{g,r}$ is of finite order, where $r \in \{0,1\}$. Then $\Fix(\varphi_{g,r}) \leq Z(G_{p,p})$. Indeed $\Fix(\varphi_{g,1}) = \{1\}$ and $\Fix(\varphi_{g,0}) = \langle x^p \rangle \cong \mathbb{Z}$.
\end{lemma}
\begin{proof}
By Lemma \ref{compatibilitylemma} and the fact that $g\sigma\tau^r(g) = 1$, the action of $\varphi_{g,r} = \iota_g \sigma \tau^r$ exchanges the vertices $v_{\langle x \rangle}$ and $gv_{\langle y \rangle}$, fixing the midpoint of the geodesic between these vertices. Because one vertex is $v_{\langle x \rangle}$ and the other is a translate of $v_{\langle y \rangle}$, the distance between these vertices is odd. Therefore the midpoint is the midpoint of an edge. This midpoint is the unique point fixed by the action of $\varphi_{g,r}$, since the endpoints of the corresponding edge are swapped, and the points fixed by an isometry of a tree form a connected subtree by Proposition \ref{geometricfixsubtree}.

By Lemma \ref{fixaction} any $g \in \Fix(\varphi_{g,r})$ fixes this midpoint and so fixes the edge, since $G_{p,p}$ acts without inversions. However the edge stabilisers are central, completing the proof of the first claim.

It is now easy to see that $\iota_g \sigma$ fixes the whole centre, and $\iota_g \sigma\tau$ fixes only the identity.
\end{proof}

\bibliographystyle{plain}
\bibliography{references}

\begin{thebibliography}{10}

\bibitem{GHMR99}
Tree actions of automorphism groups.
\newblock {\em Journal of Group Theory}, 3(2):213--223, 2000.

\bibitem{A21}
Naomi Andrew.
\newblock Serre’s property (fa) for automorphism groups of free products.
\newblock {\em Journal of Group Theory}, 24(2):385--414, 2021.

\bibitem{BH92}
Mladen Bestvina and Michael Handel.
\newblock Train tracks and automorphisms of free groups.
\newblock {\em Annals of Mathematics}, 135(1):1--51, 1992.

\bibitem{DS75}
Joan~L. Dyer and G.~Peter Scott.
\newblock Periodic automorphisms of free groups.
\newblock {\em Communications in Algebra}, 3(3):195--201, 1975.

\bibitem{F21}
Elia Fioravanti.
\newblock Coarse-median preserving automorphisms.
\newblock {\em To appear in Geometry \& Topology}, 2023.

\bibitem{GM18}
Anthony Genevois and Alexandre Martin.
\newblock Automorphisms of graph products of groups from a geometric
  perspective.
\newblock {\em Proceedings of the London Mathematical Society},
  119(6):1745--1779, 2019.

\bibitem{H01}
Allen Hatcher.
\newblock {\em Algebraic topology}.
\newblock CUP, 2001.

\bibitem{I02}
Nikolai~V. Ivanov.
\newblock Chapter 12 - mapping class groups.
\newblock In R.J. Daverman and R.B. Sher, editors, {\em Handbook of Geometric
  Topology}, pages 523--633. North-Holland, Amsterdam, 2001.

\bibitem{CHR20}
Sarah~Rees Laura~Ciobanu, Derek~Holt.
\newblock Equations in groups that are virtually direct products.
\newblock {\em Journal of Algebra}, 545:88--99, 2020.
\newblock Special Issue in Memory of Charles Sims.

\bibitem{M17}
Alexandre Martin.
\newblock {On the cubical geometry of Higman’s group}.
\newblock {\em Duke Mathematical Journal}, 166(4):707 -- 738, 2017.

\bibitem{MO11}
Ashot Minasyan and Denis Osin.
\newblock Fixed subgroups of automorphisms of relatively hyperbolic groups.
\newblock {\em The Quarterly Journal of Mathematics}, 63(3):695--712, 2011.

\bibitem{N92}
Walter~D Neumann.
\newblock The fixed group of an automorphism of a word hyperbolic group is
  rational.
\newblock {\em Inventiones mathematicae}, 110(1):147--150, 1992.

\bibitem{SS12}
J.~Stilwell and J.P. Serre.
\newblock {\em Trees}.
\newblock Springer Monographs in Mathematics. Springer Berlin Heidelberg, 2012.

\bibitem{S02}
Mihalis Sykiotis.
\newblock On the automorphisms of the fundamental group of a graph of groups
  and their fixed points.
\newblock {\em Algebra Colloquium}, 9(4):445--454, 2002.

\bibitem{V23}
Nicolas Vaskou.
\newblock Automorphisms of large-type free-of-infinity artin groups,
  arxiv:2304.06666, 2023.

\end{thebibliography}

\end{document}